\documentclass[sn-mathphys-num]{sn-jnl}

\usepackage{filecontents}

\usepackage{graphicx}%
\usepackage{multirow}%
\usepackage{amsmath,amssymb,amsfonts}%
\usepackage{amsthm}%
\usepackage{mathrsfs}%
\usepackage[title]{appendix}%
\usepackage{xcolor}%
\usepackage{textcomp}%
\usepackage{manyfoot}%
\usepackage{booktabs}%
\usepackage{algorithm}%
\usepackage{algorithmicx}%
\usepackage{algpseudocode}%
\usepackage{listings}%

\usepackage{amsmath}
\usepackage{amsfonts}
\usepackage{amsthm}
\usepackage{eurosym}
\usepackage{graphicx}
\usepackage{amssymb}
\usepackage{epstopdf}
\usepackage{topcapt}
\usepackage{latexsym}
\usepackage{mathtools}
\usepackage{graphics}
\usepackage{subfigure}
\usepackage{dsfont}
\usepackage[english]{babel}
\usepackage{color}

\usepackage{float}

\theoremstyle{thmstyleone}%
\newtheorem{theorem}{Theorem}
\newtheorem{proposition}[theorem]{Proposition}%

\theoremstyle{thmstyletwo}%
\newtheorem{remark}{Remark}%

\theoremstyle{thmstylethree}%

\begin{document}

\title[A stochastic population model with hierarchic size-structure]{A stochastic population model with hierarchic size-structure}


\author[1]{\fnm{Carles} \sur{Barril}}


\author[1,2]{\fnm{\`{A}ngel} \sur{Calsina}}


\author[1]{\fnm{J\'{o}zsef Z. } \sur{Farkas}}


\affil[1]{\orgdiv{Department de Matem\`{a}tiques}, \orgname{Universitat Aut\`{o}noma de Barcelona}, \orgaddress{\street{} \city{Bellaterra}, \postcode{08193},  \country{Spain}}}

\affil[2]{\orgdiv{Centre de Recerca Matem\`{a}tica},  \orgaddress{ \city{Bellaterra}, \postcode{08193},  \country{Spain}}}


\abstract{We consider a hierarchically structured population in which the amount of resources an individual has access to is affected by individuals that are larger,  and that the intake of resources by an individual only affects directly the growth rate of the individual.   We formulate a deterministic model,  which takes the form of a delay equation for the population birth rate.  We also formulate an in\-di\-vi\-dual based stochastic model,   and study the relationship between the two models. In particular the stationary birth rate of the deterministic model is compared to that of the quasi-stationary birth rate of the stochastic model.  
Since the quasi-stationary birth rate cannot be obtained explicitly,  we derive a formula to approximate it.   We show that the stationary birth rate of the deterministic model can be obtained as the large population limit of the quasi-stationary birth rate of the stochastic model.  This relation suggests that the deterministic model is a good approximation of the stochastic model when the number of individuals is sufficiently large.
}

\keywords{Hierarchically size-structured population,  individual based stochastic model,  quasi-stationary distribution.}


\pacs[MSC Classification]{60K40, 45G10, 92D25}

\maketitle

\section{Introduction}\label{sec1}

Many deterministic population models rely implicitly on the assumption that the number of individuals is large enough so that stochastic fluctuations can be neglected. In order for a model of this kind to be mechanistically justified, it should describe  the behaviour of an individual based stochastic model at the so-called large population limit.  To be more specific  let us examine the following example.  Consider the ``logistic'' birth and death process taking place in an area of size $A$ (we assume that individuals are well mixed in that area), which is given by the transition rates
\[
\tilde{\lambda}_k = \lambda\left(\frac{k}{A}\right)k\qquad\text{and}\qquad \tilde{\mu}_k=\mu\left(\frac{k}{A}\right)k,
\]
where {$k$ is the number of individuals and} $\lambda$ and $\mu$ are the fertility and the mortality rates of individuals respectively, which depend on the density of individuals in the population, {i.e. $k/A$}. Denote by $N_t^A$ the number of individuals of such a process. Consider the particular case in which $\lambda(n)\equiv \lambda_0$ and $\mu(n)=\mu_0+\alpha n$, with $\lambda_0>\mu_0$. From the classical results derived in \cite{kurtz1970} it follows that, for a given time window $[0,\hat{t}]$, the trajectories of $n_t^A:=N_t^A/A$ with $P(N_0^A=\lfloor n_0 A\rfloor ) = 1$ are, for large enough $A$, arbitrarily close to the orbit $n(t)$ of the initial value problem
\begin{equation}\label{IVPlogistic}
\left\{\begin{array}{l}
n'(t)=\lambda_0 n(t) - (\mu_0+\alpha n(t))n(t)\\
n(0)=n_0
\end{array}
\right. ,
\end{equation}
where $n(t)$ refers to the population density at time $t$ (its units are individuals per area, the same ones of the rescaled process $n_t^A$). More precisely, for all $\varepsilon>0$
\begin{equation}\label{DetLimitScalar}
\lim_{A\rightarrow\infty} P\left(\left|N_t^A/A-n(t)\right|>\varepsilon \;|\; N_0^A=\lfloor n_0 A\rfloor\right) = 0,\qquad\forall t\in [0,\hat{t}].
\end{equation}
This important relation implies that the differential equation in \eqref{IVPlogistic}, which is nothing but the logistic equation $n'(t)=r(1-n(t)/\kappa)n(t)$ with $r=\lambda_0-\mu_0$ and $\kappa=r/\alpha$, is a mechanistically justified and valid model. As a consequence it is a model useful to predict the behaviour of the underlying stochastic model at the large population limit. It is worth noticing, however, that for any fixed $A$, no matter how large it is, there is an important disagreement between $n(t)$ and $n_t^A$ for large $t$, due to the fact that $0$ is an absorbing state of the stochastic system.  In particular,  it can be shown that for all $\varepsilon>0$ the limit of $P(n_t^A>\varepsilon)$ as $t$ tends to infinity is 0, whereas $n(t)$ tends to $\kappa$ as $t$ tends to infinity. This discrepancy is inherently associated to the nature of population models,  since no individuals can be born from an extinct population.

As shown in \cite{kurtz1970}, the theory to derive ordinary differential equations as limits of scaled Markov processes can be applied not only to competition systems of one species as the example above, but essentially to any Markov process where the states are discrete and refer to the number of individuals  of each species in the system. Such a relation between ODEs and Markov processes, however, does not help to study the stochastic dynamics when the area $A$ is not large enough. The question whether $A$ is large enough, or more precisely the study, as a function of $A$, of the time-scale during which the rescaled stochastic process is close to the deterministic ones,  is addressed in different studies see,  for instance, \cite{chazottes2016, naasell2002, andersson2000}.  When the process has $0$ as the only absorbing state, two useful concepts to tackle this question are the probability of extinction of the process (as a function of time) and the quasi-stationary distribution of the process,  defined as the asymptotic distribution of the process conditioned to non-extinction.

If the Markov process takes values in an infinite dimensional space, some results in the spirit of \eqref{DetLimitScalar} can be found. Infinite dimensional Markov processes are important since they arise naturally when dealing for example with continuously structured population models. In these cases the state variable is a function of the structuring variable instead of a finite dimensional vector (see \cite{diek1998li,  diek2001nonlin} for a deterministic formulation of these models).  In \cite{metz2013daphnias, carmona2018some} physiologically structured population models close to the one analysed in the present paper are considered. There individuals are assumed to change their physiological state (such as size or age) deterministically between jump events, where jumps occur according to some probability measures and describe stochastic events such as births, deaths or transitions between different individual classes. Large population limit results are given, which relate the stochastic system to a deterministic PDE.  These works extend previous studies in which large population limits were established for stochastic age-structured populations, such as \cite{oelschlager1990limit,tran2008large}. Similar techniques are used in \cite{bansaye2015} to stablish analogous results for a different kind of stochastic structured population models. There, the population is structured by a genetically determined phenotype, which is assumed to be constant during whole life (new phenotypes arise due to mutation events during reproduction). See also \cite{guerrero2015} for an application in which deterministic limits of stochastic systems have been applied to link the dynamics of a cellular population with a submodel describing cell metabolism.


Individual based stochastic size-structured populations have been analysed through approaches not directly related to generalisations of \cite{kurtz1970}, such as in \cite{allen2009}. There a deterministic size structured population model is expanded into a stochastic PDE which takes into account random fluctuations in the number of births, deaths and the way individuals grow. The methodology is based on the theory developed in \cite{allen2008}, which starts with the construction of a discrete (in time and size) stochastic model. These results could be potentially helpful to address some of the questions we pose throughout the paper. However, the {connection} is not completely clear since stochastic PDEs and stochastic models defined at the individual level are different in general. 

In this article we present an individual based stochastic version of the deterministic hierarchically size-structured population model we recently studied in \cite{BCDF2023}. The hierarchical structure is present due to the assumption that the amount of resources an individual has access to is affected (only) by those individuals that are larger.  As far as we know this type of size-structured contest competition has not been modelled and analysed previously using stochastic birth and death events.
In contrast, the deterministic,  partial differential equation formulation of hierararchical age and size-structured populations is well studied.   We mention the possibly earliest such work \cite{cushing1994}, where a hierarchical age-structured contest competition model was introduced and studied.  In \cite{2006basic} existence of solutions of a very general class of hierarchical size-structured population models was established.  The more recent paper \cite{Farkas2010} focused on the qualitative analysis of hierarchical size-structured models via spectral methods.  Most notably,  the emergence of measure valued solutions for quasilinear hierarchical size-structured population models was demonstrated in \cite{AD,AI} using the traditional PDE formulation.

Here we formulate the deterministic model studied in \cite{BCDF2023} is as a renewal equation for the population birth rate per unit of area. The biological assumptions of the model are:
\begin{itemize}
\item individuals are born with the same minimal size $x_m$,
\item the growth rate of an individual is positive and depends on the amount of larger individuals, i.e. the growth rate of an individual with size $x$ at time $t$ is given by
\begin{equation}\label{growthSimple}
g\left(\int_x^\infty u(t,y)dy\right),
\end{equation}
where $g$ is a positive function and $u(t,x)$ is the population density with respect to size at time $t$ and per unit of area (so that $\int_x^\infty u(t,y)dy$ is the number of individuals with size larger than $x$ per unit of area, i.e. the spatial density of individuals that are larger than $x$), 
\item the fertility rate of an individual of size $x$ is given by $\beta(x)$,
\item the mortality rate of all individuals is equal to $\mu$ (a positive constant).
\end{itemize}
Since all individuals are born with the same size, the positiveness of the growth rate implies that an individual is larger than some other individual if and only if it is older than that other individual. Therefore the growth rate of an individual with age $a$ at time $t$ can be given as
\[
g\left(\int_a^\infty b(t-\alpha)e^{-\mu \alpha}d\alpha\right)
\] 
where $b(t)$ is the population birth rate at time $t$ per unit of area, and hence $b(t-\alpha)e^{-\mu \alpha}$ gives the density of individuals with age $\alpha$ at time $t$ per unit of area (the individuals born at time $t-\alpha$ that have survived $\alpha$ units of time, i.e. until time $t$). Then the renewal equation for the population birth rate reads
\begin{equation}  \label{scalar2}
b(t)= \int_0^{\infty} \beta \left( x_m + \int_0^a g \left(
\int_{a}^{\infty} e^{-\mu(\tau-a+s)} b_t(-s) \mbox{d}s
 \right) \, \mbox{d}\tau \right)\,\, e^{-\mu a}b_t(-a) \,
\mbox{d}a,
\end{equation}
where
\begin{equation}
b_t(\theta):=b(t+\theta).
\end{equation}
Notice that \eqref{scalar2} can be understood directly from the description of the physical situation: it states that the birth rate at time $t$ is given by the \textit{addition} for the age $a$ of the mothers, with density equal to the birth rate at time $t-a$ times their survival probability, their size specific fertility with size given by the birth size plus the integral with respect to $\tau$ of the individual growth rate at age $\tau$, which depends on how many individuals were larger (equivalently older) than them at that age:  this is the integral of the density of individuals of age $s>a$, i.e. born at time $t-s<t-a$, that survived until time $t-a+\tau$, which is precisely the time when the mother born at time $t-a$ has age $\tau$.  Notice that from $t-s$ to $t-a+\tau$ there are $\tau-a+s$ units of time, in accordance to the survival factor inside the integral.

In \cite{BCDF2023} it is shown that if $g$ is decreasing and $\beta$ is increasing, then (\ref{scalar2}) can have at most one non-trivial stationary birth rate (i.e. a constant value $\bar{b}>0$ satisfying (\ref{scalar2})). If $g(z)\rightarrow 0$ as $z\rightarrow \infty$, such a non-trivial equilibrium exists if and only if the basic reproduction number $R_0$ (i.e. the expected number of offspring of an individual in the extinction environment) is larger than one. In terms of the model ingredients such a condition takes the form:
\begin{equation}\label{R0eq}
R_0:=\int_0^\infty \beta(x_m + g(0)a)e^{-\mu a} da > 1.
\end{equation}
The stochastic version we introduce and study differs from the deterministic model in that birth and mortality ``deterministic'' rates are replaced by its corresponding probability rates for an individual to be born or to die.

The goal of this paper is to analyse the asymptotic distribution of the stochastic model and study whether the deterministic model of \cite{BCDF2023} is a good approximation of the stochastic model presented here. To this end we compare the (unique) positive stationary birth rate of the deterministic model with the (unique) quasi-stationary birth rate of the stochastic model, and we show that the former can be recovered as a certain limit of the latter.

\begin{figure}
\includegraphics[width=\textwidth]{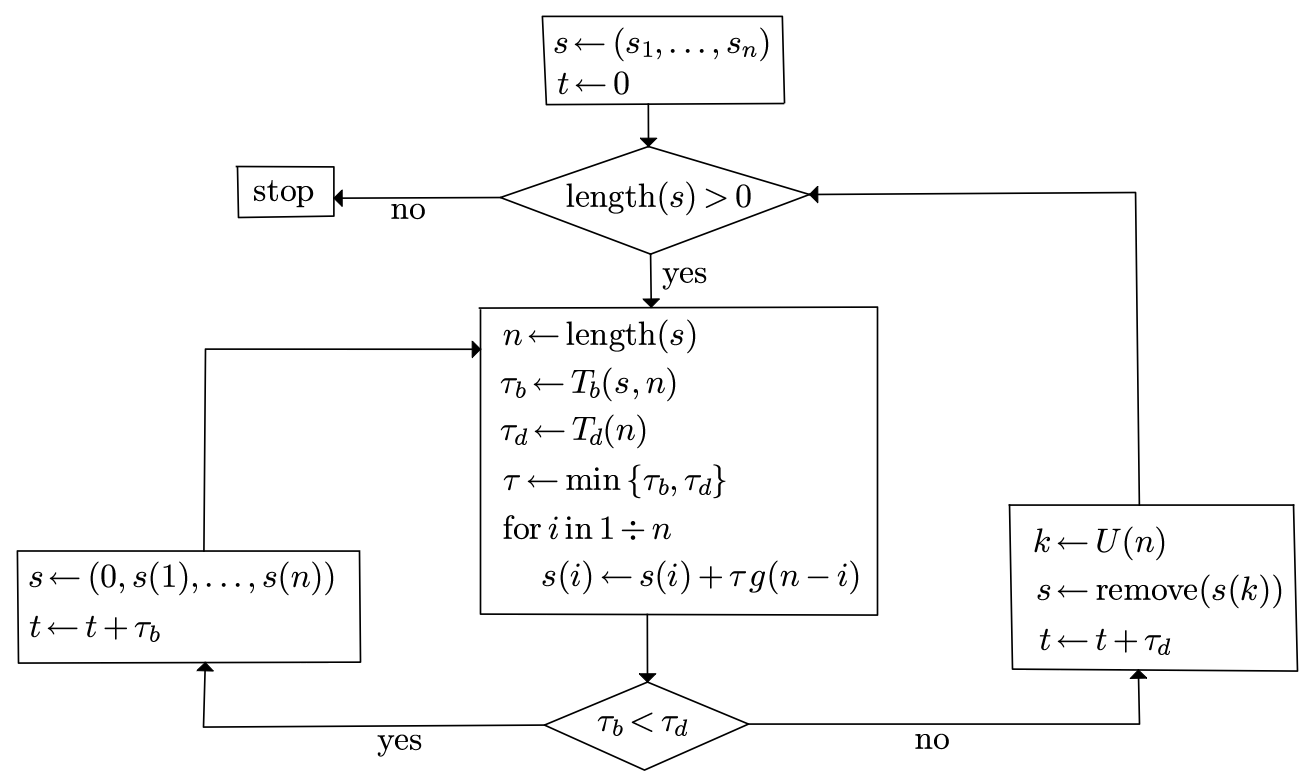}
\caption{\label{programDiagram} Flow diagram of the program that computes realisations of the stochastic process described in Section \ref{sthocSection}. The state of the system is represented by a vector of sizes. Sizes are ordered from small to large (because older individuals are always larger than younger ones). The function $T_d(n)$ gives a random number exponentially distributed with rate $n\mu$. It simulates the time until the next death.  The function $T_b(s,n)$ is a random number distributed according to the density $f_{T_b}(\tau)=B(\tau) \exp\left(-\int_0^\tau B(\theta)d\theta\right) \mathds{1}_+(\tau)$ with $B(\tau)=\sum_{i=1}^n \beta(s(i)+\tau g(n-i))$. It simulates the time until the next birth.  The function $U(n)$ gives a random index from $\{1,2,..,n\}$ uniformly distributed, and $\text{remove}(s(k))$ returns vector $s$ with the $k${-th} component removed, {and} this process simulates the individual who dies.} 
\end{figure}

\section{Stochasticity and population size}\label{sthocSection}
Most biological populations have a discrete number of agents and are inherently stochastic. Hence, continuous and deterministic equations modelling them (such as equation (\ref{scalar2})) must be understood as certain approximation of the stochastic, individual based process that gives an accurate description of the population dynamics.

In this section we compare the renewal equation (\ref{scalar2}) with the stochastic system that it supposedly represents, which models the dynamics of a population with a finite number of individuals that are well mixed in an environment of area $A$. 
Here we assume that individuals live in a two-dimensional environment (as do terrestrial animals) to ease the explanation, but we could alternatively use volumes instead of areas if the population inhabited a three-dimensional environment. In any case, the role played by the parameter $A$ is related to the force of competition: the more space individuals have, the less competition and hence the larger the population can grow. The details of the process are the following:

\begin{itemize}
\item[$\bullet$] The state of the system is a vector $(t_1,t_2,\dots,t_M)$ representing the times of each birth event that has taken place in the past, whose components satisfy $t_1<t_2<\cdots<t_M$ so that the $k$-th component of the vector is the time of the $k$-th most ancient birth event. Its dimension (the number of accumulated births denoted by $M$) increases in time due to the new birth events.

\item[$\bullet$]The mortality rate of individuals is denoted by  $\mu$, which means that the probability that an individual alive at time $t$ dies in the time window $[t,t+dt]$ divided by $dt$ tends to $\mu$ as $dt$ tends to 0. This is equivalent to say that the life span of each individual is independent of the other individuals and follows an exponential random variable with parameter $\mu$. In particular, the individual born at time $t_k$ dies at time $\bar{t}_k=t_k+X_k$ with $X_k\sim \text{Exp}(\mu)$.  The amount of individuals that are alive at time $t$ is thus
\[
N(t)=\sum_{k=1}^{M(t)}\mathds{1}_+(\bar{t}_k-t),
\]
where the function $\mathds{1}_+$ denotes the Heaviside theta function.

\item[$\bullet$] The birth rate is given by
\[
B(t)=\sum_{i=1}^{N(t)} \beta(s_i(t)),
\]
where $s_i(t)$ is the size of the $i$-th smallest individual at time $t$. The above expression implies that the probability that a birth occurs in a time window $[t,t+dt]$ divided by $dt$ tends to $B(t)$ as $dt$ tends to 0.

\item[$\bullet$] All individuals are born with size $x_m$  (the minimal size)  and the growth rate of the $i$-th smallest individual is determined by 
\begin{equation}\label{growthRateSto}
s_i'(t)=
{
g\left(\frac{1}{A}\sum_{j=i+1}^{N(t)} 1  \right)
= \,}
g\left(\frac{N(t)-i}{A} \right) ,
\end{equation}
so that the growth of an individual depends on the number of larger individuals (recall that the positivity of $g$ and the fact that everybody is born with the same size implies that being larger is equivalent to being older, so that the number of larger individuals  reads as well as the number of older individuals). Since the function $g$ is only evaluated at numbers of the form $j/A$ with $j$ a natural number, we define $g_j:=g(j/A)$ for all $j\in\mathbb{N}$. {Notice also that the argument of $g$ is a density of individuals (recall that $A$ is the area inhabited by the individuals), so that the individuals growth rate in the determinstic model is equivalent to the individuals growth rate in the stochastic one.}     
\end{itemize}

The functional parameters $\beta$ and $g$, and the parameter $\mu$ coincide with the ones used in equation (\ref{scalar2}), since their biological meaning is independent on whether the model is deterministic or stochastic. Notice that $B(t)$ (as well as $M(t)$, $N(t)$ and $s_i(t)$ for $i\in\{1,\dots,M(t)\}$) depend on the parameter $A$. When needed, this dependence is emphasized by writing $B(t;A)$. Similarly we write $B(t;v,A)$ with $v=(t_1,t_2,\dots,t_M)$ to emphasize the dependence of the birth rate on the vector of initial birth events.

To simplify the notation, from now on we take,  without loss of generality,   $x_m=0$. {This is not a restriction because the model can be formulated in terms of the auxiliary structuring variable $y=x-x_m$ (if $x$ denotes size, $y$ would denote the difference between size and the size at birth). However, let us note that for some biological systems, $x_m=0$ is a good approximation. Indeed, for a model describing a height-structured tree population with $x$ referring to the height of trees, then it makes sense to consider $x_m=0$}. In Figure \ref{programDiagram} a flow diagram of the program used to simulate this process is drawn (where we work with the vector of sizes instead of the vector of birth events for computational reasons).

As commented in the introduction, a key question about the relation between the solution of \eqref{scalar2} and the scaled process $B(t;A)/A$ is whether the former can be obtained as the (deterministic) limit of the latter as $A$ tends to infinity and the initial vector of birth events is scaled properly with $A$, by choosing, for instance, $M(0;A):=\lfloor A \int_{-\infty}^0 b_0(\theta)d\theta \rfloor$ realisations, stored in a vector $\theta(A)=(\theta_1,\dots,\theta_{M(0;A)})$ in increasing order, of a random variable with a density function proportional to $b_0\in L^1(-\infty,0)$ (i.e. the initial history of birth rates per unit of area used in the renewal equation \eqref{scalar2}). That is, to show in analogy to (\ref{DetLimitScalar}) that for all $\varepsilon>0$ one has 
\[
\lim_{A\rightarrow\infty} P\left(\left|\frac{B(t;\theta(A),A)}{A}-b(t)\right|>\varepsilon\right) = 0,\qquad\forall t\in [0,\hat{t}],
\]
where $\hat{t}$ is a fixed time horizon (which can be arbitrarily large) and $B(t;\theta(A),A)$ is the birth rate process obtained when the vector $\theta(A)$ of initial birth events is constructed randomly as just explained. The results given in \cite{metz2013daphnias, carmona2018some} seem promising to tackle this question. However, they consider the population density as the state variable instead of the birth rate history, so that applying them would not be straightforward in our setting. 

In this work we don't provide a proof of the above relation, but we give a result pointing towards this conclusion based on the ``equivalent'' asymptotic behaviour of the two systems in the large population limit. We take for guaranteed that, if $g$ is decreasing and $\beta(x)=\beta_0 x$, the {expected} birth rate conditioned to non-extinction, defined {as the conditional expectation}
\[
\tilde{B}(t;v,A):=\mathds{E}(B(t;v,A)\,|\,B(t;v,A)>0),
\]
converges towards a quasi-stationary birth rate $\bar{B}_*(A)$, i.e.
\[
\lim_{t\rightarrow\infty} \tilde{B}(t;v,A) = \bar{B}_*(A),
\]
for all (non-empty) vector $v$ of initial birth events. This is certainly a non-trivial result whose validity needs to be addressed rigorously. In Section \ref{exUniqQSBR} we derive a heuristic approximation of $\bar{B}_*(A)$, denoted by $\bar{B}(A)$. The accuracy of the approximation increases with $A$, and we show numerically (in Section \ref{numericalSection})  that $\bar{B}(A)$ is already a good approximation of $\bar{B}_*(A)$ for small values of $A$.

In Section \ref{formalComparison} we prove that if $R_0>1$ (see (\ref{R0eq})), so that a non-trivial stationary birth rate per unit of area of the renewal equation (\ref{scalar2}) does exist, denoted by $\bar{b}$, then
\begin{equation}\label{StoToDetBRlimit}
\lim_{A\rightarrow \infty} \frac{\bar{B}(A)}{A} = \bar{b}.
\end{equation}

\subsection{Existence and characterisation of stationary birth rates}\label{exUniqQSBR}

Here we compute some expected properties the population has at a given time, let us say at $t=0$, under the assumption that the birth events before that time followed a Poisson process with parameter $\bar{B}(A)$ (i.e. we assume that in the past the birth events were determined by a potential quasi-stationary birth rate whose inter event times are described by independent random variables exponentially distributed with parameter $\bar{B}(A)$). These quantities, expressed in terms of $\bar{B}(A)$, allow us to write a consistency equation $\bar{B}(A)$ must satisfy in order to be quasi-stationary. We deduce, therefore, that the solutions of this equation with unknown $\bar{B}(A)$ give the quasi-stationary birth rate of the stochastic system.

\begin{remark}\label{remarkBAprox}
We assume that the birth events prior to $t=0$ follow a Poisson process of parameter $\bar{B}(A)$ because this allows us to obtain a closed formula for an approximation of the quasi-stationary birth rate defined above as $\bar{B}_*(A)$. The birth events of the real system once it reaches a quasi-stationary regime, however, deviate from this assumption: lapses of inter birth events are not homogeneously exponentially distributed and there are correlations between nearby inter birth events lapses (because the time between two consecutive birth events depends on the actual population, and the population at a given time is highly correlated with the population at nearby times). The value $\bar{B}(A)$ we derive thanks to that assumption is, therefore, an approximation of the real quasi-stationary birth rate. It seems reasonable to conjecture, however, that
\begin{equation}\label{equivBaprox}
\lim_{A\rightarrow \infty} \frac{\bar{B}_*(A)}{\bar{B}(A)}=1,
\end{equation}
i.e. that the larger $A$ the closer $\bar{B}(A)$ is to $\bar{B}_*(A)$ in relative terms (since both of them diverge as $A$ tends to infinity). In particular, if this relation holds true, then proving \eqref{StoToDetBRlimit} for the approximation would imply the equivalent result for the real quasi-stationary birth rate. In this work we draw conclusions by assuming \eqref{equivBaprox} without providing a proof of it. 
\end{remark}
We start computing the distribution of individuals at $t=0$. Specifically, let $\bar{N}(A)$ give the number of individuals when the birth events in the past are given by the random sequence
\[
(...,t_{-3},t_{-2},t_{-1}),
\]
where $t_{-1}=X_1$ and $t_{-k-1}=t_{-k}-X_{k+1}$ with $\{X_i\}_{i\in\mathbb{N}}$ mutually independent random variables satisfying $X_i\sim \text{Exp}(\bar{B}(A))$ for $i\in \mathbb{N}$.

Since the birth rate is assumed to be $\bar{B}(A)$ and the mortality rate of individuals is constant and equal to $\mu$, the total population has evolved in the past according to a birth and death process with $(\lambda_k,\mu_k)=(\bar{B}(A),k\mu)$ for all $k\in \mathbb{N}$. Therefore, the number of individuals at time 0 is distributed according to the stationary distribution of this birth and death process, i.e. the random variable $\bar{N}(A)$ has density
\begin{equation}\label{pkDef}
p_k:=P(\bar{N}(A)=k)=\left(\frac{\bar{B}(A)}{\mu}\right)^k \frac{1}{k!}e^{-\bar{B}(A)/\mu},\qquad k\in\mathbb{N}\cup\{0\}.
\end{equation}
\begin{remark}
Notice that, since $P(\bar{N}(A)=0)>0$, the probabilities \eqref{pkDef} cannot represent the distribution of individuals at the quasi-stationary distribution associated to the process detailed at the beginning of Section \ref{sthocSection}.  (i.e. the distribution of individuals to which the system converge conditioned to non-extinction of the population). We conjecture, however, that $\bar{N}(A)$ is close to such distribution and that the larger $A$ is the closer it is (which makes sense since the larger $A$, the larger $\bar{B}(A)$ as we are about to see, and hence the smaller the probability that $\bar{N}(A)$ takes the value 0).
\end{remark}
Now let us focus on the expected size an individual has at a given age. This size depends on the number of other individuals present in the population when it is born. For instance, if an individual is {introduced into an empty habitat (corresponding to the extinction steady state)},  then the individual will grow at a constant speed $g_0$. Specifically, its size will be given by (recall that $x_m=0$)
\[
s_0(\tau)=g_0 \tau,
\]
with $\tau$ being the age of the individual. If, instead, there is already one individual in the population when the new individual is born, then the size of this new individual at age $\tau$ depends on the time the older individual dies. It is no longer a certain value but a random variable. The expected size at age $\tau$ of an individual born in a population with one other individual is
\[
s_1(\tau)=g_1 \left( E(Y_1^1 | Y_1^1<\tau)P(Y_1^1 <\tau)+\tau P(Y_1^1>\tau) \right)+g_0 E(\tau-Y_1^1 | Y_1^1 <\tau) P (Y_1^1 <\tau),
\]
where $Y_1^1$ is a random variable that gives the time at which the older individual dies. Notice that $g_1$ is multiplied by the expected duration of the lapse within $[0,\tau]$ during which the older population has $1$ individual and $g_0$ is multiplied by the expected duration of the lapse within $[0,\tau]$ during which the older population has zero individuals. This observation allows us to generalise the above expression to give the expected size at age $\tau, s_k(\tau),$ of an individual born when there are already $k$ other individuals in the population. To do so we first define the random variables:
\begin{equation}\label{YjkDef}
Y_j^k := \text{time at which the }j\text{ death occurs in a population of }k\text{ individuals }.
\end{equation}
For example, $Y_1^1$ is an exponential random variable with rate $\mu$. The dependencies and distributions of $Y_j^k$ in terms of $j$ and $k$ can be found in Appendix \ref{apendixSizeAtAge}, but here notice that  $Y_{j}^k \leq Y_{j+1}^k$ by definition. In order to deal with more compact formulas, the particular cases in which $k=0$, $j=0$ and $j>k$ are defined directly as $Y_0^0\equiv \infty$, $Y_0^k\equiv 0$ if $k>0$ and  $Y_j^k\equiv \infty$ if $j>k$. Then, let $l_{k-j}^k(\tau)$ be the expected duration of the lapse within $[0,\tau]$ during which the older population (with $k$ individuals at time $0$) has $k-j$ individuals (i.e. the $j$ death has occurred but not the $j+1$). This value is given in terms of $Y_{j+1}^k$ and $Y_{j}^k$ as
\begin{equation}\label{lapse}
\begin{aligned}
l_{k-j}^k(\tau)= & E(Y_{j+1}^k-Y_{j}^k | Y_{j+1}^k < \tau)P(Y_{j+1}^k < \tau) \\  
  & +E(\tau-Y_{j}^k | Y_{j}^k <\tau, Y_{j+1}^k >\tau) P(Y_{j}^k <\tau, Y_{j+1}^k >\tau),
\end{aligned}
\end{equation}
so that the function $s_k(\tau)$ is given as
\begin{equation}\label{sAgek}
s_k(\tau)=\sum_{j=0}^k g_{k-j} l_{k-j}^k(\tau) .
\end{equation}

By combining the set of $p_k$ and the set of $s_k$, the expected size of an individual at age $a$ is
\begin{equation}\label{sumSAgek}
\bar{s}(a)=\sum_{k=0}^\infty p_k s_k(a),
\end{equation}
and it can be shown (see Appendix \ref{apendixSizeAtAge}) that the above equation reduces to
\begin{equation}\label{eqSizeAtAge1}
\bar{s}(a)=\frac{1}{\mu}
\sum_{m=0}^\infty g_m \frac{1}{m!}
	\left(
	\Gamma\left(m,\frac{\bar{B}}{\mu}e^{-\mu a}\right)
	-
	\Gamma\left(m,\frac{\bar{B}}{\mu}\right)
	\right)
\end{equation}
where $\Gamma$ is the upper incomplete gamma function and where we omitted the dependence of $\bar{B}$ on $A$. Notice that the analogue of function $\bar{s}$ in the deterministic model is a function, denoted here as $\bar{s}_\text{det}$, that gives the size of individuals in terms of their age once the system has reached the stationary birth rate $\bar{b}$ of (\ref{scalar2}). A formula for $\bar{s}_\text{det}$ is (see (4.2) in \cite{BCDF2023})
\begin{equation}\label{rational_size1}
\bar{s}_\text{det}(a;\bar{b})=\int_0^a g\left(\frac{\bar{b}}{\mu}e^{-\mu\tau}\right)d\tau .
\end{equation}
Since $\bar{s}$ and $p_k$ depend on the birth rate $\bar{B}$, from now on we write $s(\tau;\bar{B})$ and $p_k(\bar{B})$. The idea now is to impose an equation for the birth rate $\tilde{B}$. Such equation is analogous to equation (4.1) in \cite{BCDF2023}, namely
\[
1 =  \int_0^\infty \beta(\bar{s}_\text{det} (a;\bar{b})) e^{-\mu a}da,
\]
which, if $\beta(x)=\beta_0 x$,  reads as
\begin{equation}\label{rationalbirthrate_ImplicitForm}
1 = \beta_0 \int_0^\infty \bar{s}_\text{det} (a;\bar{b}) e^{-\mu a}da.
\end{equation}
Indeed, this equation is obtained by imposing that the expected offspring of an individual is equal to 1 (so that the population stays constant from a statistical point of view). The same can be accomplished in the stochastic framework by noticing that the expected offspring of an individual born in a population with $k$ older individuals is
\[
\int_0^\infty \beta(s_k(a))e^{-\mu a} da ,
\]
so that, taking into account all the possible scenarios in which an individual can be born, the mentioned equation reads
\begin{equation}\label{BEqS	thGeneral}
1=\sum_{k=0}^\infty p_k(\bar{B})\int_0^\infty \beta(s_k(a))e^{-\mu a} da ,
\end{equation}
which becomes
\begin{equation}\label{BEqSth1}
1=\beta_0 \int_0^\infty \bar{s}(a;\bar{B})e^{-\mu a} da, 
\end{equation}
in the particular case that $\beta$ has the form $\beta(x)=\beta_0 x$. {Choosing the relation $\beta(x)=\beta_0 x$ between fertility and size is a notable restriction on the generality of the model. It would certainly be better to prove the following results by assuming only that $\beta(x)$ is increasing. It turns out, however, that for the particular relation $\beta(x)=\beta_0 x$ many computations become explicit and results can be obtained more easily. Nevertheless, the rationale behind such a particular case could help to deal with the general case.}

{From \eqref{BEqS	thGeneral} it then follows} that there are as many quasi-stationary birth rates as positive solutions equation \eqref{BEqS	thGeneral} has for the unkown $\bar{B}$. As shown next, if $\beta(s)=\beta_0 s$ and $g$ reflects a scenario of pure competition (no Allee effects), then \eqref{BEqSth1} has at most one positive solution, and such a positive solution exists if and only if $R_0>1$ (see \eqref{R0eq}). Notice that under these assumptions 
$$R_0=\displaystyle\int_0^\infty \beta_0 g(0)a e^{-\mu a}da=\beta_0 g(0)/\mu^2.$$
\begin{theorem}\label{BbarFunctionThm}
Assume that $\beta(s)=\beta_0 s$ and $\{g_m\}_{m\in\mathbb{N}}$ is decreasing and $g_m$ tends to 0 as $m\rightarrow \infty$. If $R_0=\beta_0 g_0/\mu^2>1$, then \eqref{BEqSth1} has one and only one solution $\bar{B}>0$. If $R_0\leq 1$, then \eqref{BEqSth1} has no positive solutions.  
\end{theorem}
\begin{proof}
Define, for $B>0$ and $A>0$,
\[
R(B,A):=\beta_0 \int_0^\infty \bar{s}\left(a;B,A\right)e^{-\mu a} da
\]
so that \eqref{BEqSth1} reads $R(\bar{B},A)=1$ for a fixed $A$. Recall that $g_m$ was defined as $g(m/A)$. In the following we write $g_m(A)$ to stress the dependency of $g_m$ on $A$. Notice that, using \eqref{eqSizeAtAge1},
\begin{equation}\label{RBAform}
\begin{aligned}
R(B,A)& =\frac{\beta_0}{\mu}\sum_{m=0}^\infty \frac{g_m(A)}{m!} \int_0^\infty \int_{\frac{B}{\mu}e^{-\mu a}}^\frac{B}{\mu} x^{m-1}e^{-x}dx e^{-\mu a}da \\
& = \frac{\beta_0}{\mu}\sum_{m=0}^\infty \frac{g_m(A)}{m!} \int_0^{\frac{B}{\mu}}\int_{-\frac{1}{\mu}\ln\frac{\mu}{B}x}^{\infty} e^{-\mu a}da x^{m-1}e^{-x}dx  \\
	& = \frac{\beta_0}{\mu}\sum_{m=0}^\infty \frac{g_m(A)}{m!} \int_0^{\frac{B}{\mu}}\frac{1}{B}x^m e^{-x}dx = \frac{1}{\frac{B}{\mu}} \int_0^{\frac{B}{\mu}}f(x,A) dx,
	\end{aligned}
	\end{equation}
	with
	\[
	f(x,A):=\frac{\beta_0}{\mu^2}\sum_{m=0}^\infty \frac{g_m(A)}{m!}x^m e^{-x}.
	\]
	From \eqref{RBAform}
	 it follows $\lim_{B \rightarrow 0}R(B,A)=f(0,A)=R_0$ and $\lim_{B\rightarrow\infty} R(B,A)=0$. To see the second claim use the fact that $f(x,A)$ tends to 0 when $x$ goes to infinity: given $\varepsilon>0,$ one can choose $M$ such that $g_m < \frac{\varepsilon}{2}$ for $m > M$ and bound  
	 $$
	 0 < \frac{\mu^2}{\beta_0}f(x,A) \leq g_0 e^{-x} \sum_{m=0}^{M} \frac{x^m}{m!} + g_M e^{-x} \sum_{m=M+1}^\infty,  \frac{x^m}{m!} < \varepsilon
	 $$
	 for $x$ large enough since the second addend is less than $\frac{\varepsilon}{2}$ for any $x$.
	 \\
	 To conclude the proof we notice that $R(B,A)$ is strictly decreasing with respect to its first variable since by  \eqref{RBAform} 
	 it is the average of $f(\cdot,A)$ on the interval $[0, \frac{B}{\mu}]$ and $f(\cdot,A)$ is a strictly decreasing function. Indeed, we can compute
	 $$
	 \frac{\partial}{\partial x} f(x,A) = \frac{\beta_0}{\mu^2} e^{-x} \sum_{m=0}^\infty \frac{g_{m+1}(A)-g_{m}(A)}{m!} x^m < 0,
	 $$
	since $g_m$ is decreasing with respect to $m$.   
\end{proof}

Since the conditions in the previous statement do not depend on the value of $A$, the existence of a quasi-stationary birth is independent of $A$, and in particular a function $\bar{B}(A)$ of positive quasi-stationary birth rates does exist provided that $R_0>1$. Notice also that with $\bar{B}$ known, one can give (under the assumption that the birth events in the past are determined by $\bar{B}$) the expected total population, that is $\bar{B}/\mu$, as well as the size distribution of the individuals:
\begin{equation}\label{sizeDistSth}
\bar{U}(x)=\frac{\bar{B}}{\mu} \nu(\bar{s}^{-1}(x;\bar{B}))\frac{d}{d x}\bar{s}^{-1}(x;\bar{B}),
\end{equation}
where $\nu$ is the distribution of the individuals with respect to their age, which in this case is $\nu(a)=\mu e^{-\mu a}$. This is the analogue version (after normalising by $A$) of the size distribution associated to the steady birth rate $\bar{b}$ of the deterministic renewal equation, given by {(with $\bar{s}_\text{det}$ defined in \eqref{rational_size1})} 
\begin{equation}\label{sizeDistDet}
\bar{u}(x)=\frac{\bar{b}}{\mu}\nu(\bar{s}_\text{det}^{-1}(x;\bar{b}))\frac{d}{d x}\bar{s}_\text{det}^{-1}(x;\bar{b})
\end{equation} 
{(see (5.18) in \cite{BCDF2023} for its deduction)}.

\section{Formal comparison between the deterministic and the stochastic models}\label{formalComparison}

In this section we prove \eqref{StoToDetBRlimit}, i.e. that the stationary birth rate $\bar{b}$ of the deterministic model can be obtained as the limit as the area parameter $A$ tends to infinity of the quasi-stationary birth rate $\bar{B}(A)$ of the stochastic model after scaling it by $A$ (see Remark \ref{remarkBAprox}). To do so, we must assume $R_0> 1$ so that $\bar{b}$ and $\bar{B}(A)$ do exist and are uniquely defined, as proven in \cite{BCDF2023} for $\bar{b}$ and in Theorem \ref{BbarFunctionThm} for $\bar{B}(A)$. {Notice that the equation for the stationary birth rate $\bar{b}$ of the deterministic renewal equation \eqref{scalar2} is obtained by imposing that the birth history is constant, which gives (under the assumption $\beta(x)=\beta_0 x$ and $x_m=0$)
\begin{equation}\label{bbarDet}
1 = \int_0^{\infty} \beta_0 \int_0^a g\left(\frac{\bar{b}}{\mu}e^{-\mu\tau}\right) \, \mbox{d}\tau \,\, e^{-\mu a} \,
\mbox{d}a = \beta_0 \int_0^{\infty} \bar{s}_\text{det}(a;\bar{b}) \,\, e^{-\mu a} \,
\mbox{d}a,
\end{equation}
so that $\bar{b}$ is defined implicitly as the unique solution of that equation.
} 

\begin{theorem}\label{ThmAsyLimit} Let $\beta(x)=\beta_0 x$ and $g$ be decreasing and such that $g(z)\rightarrow 0$ as $z\rightarrow 0$. Let $R_0>1$, with $R_0$ defined in \eqref{R0eq} (given in the statement of Theorem \ref{BbarFunctionThm}). Let $\bar{b}$ be the stationary birth rate of \eqref{scalar2} and let $\bar{B}(A)$ be the solution of \eqref{BEqSth1}. Then \eqref{StoToDetBRlimit} holds, i.e. 
\[
\bar{b}=\lim_{A\rightarrow \infty} \frac{\bar{B}(A)}{A}.
\]
\end{theorem}

\begin{proof}

First notice that \eqref{StoToDetBRlimit} holds if
\begin{equation}\label{limitBtilde}
\lim_{{A\rightarrow \infty}} R(bA,A) = \lim_{{A\rightarrow \infty}} \beta_0 \int_0^\infty \bar{s}\left(a;b A\right)e^{-\mu a} da=l\quad\text{where}\quad
\left\{
\begin{array}{l}
l>1 \text{ if } b<\bar{b}\\
l=1 \text{ if } b=\bar{b}\\
l<1 \text{ if } b>\bar{b}
\end{array}
\right.,
\end{equation}
and where $R(B,A)$ is defined in \eqref{RBAform}. Indeed, let us assume that there exists an  $\varepsilon > 0$ and a sequence $A_n \rightarrow \infty$  such that $\frac{\bar{B}(A_n)}{A_n} > \bar{b} + \varepsilon$, i.e. $\bar{B}(A_n) > (\bar{b} + \varepsilon)A_n.$  Since $R$ is a strictly decreasing function of its first argument, we have
$$
1 = R(\bar{B}(A_n),A_n) <  R((\bar{b} + \varepsilon)A_n,A_n) \rightarrow l < 1,
$$
a contradiction. The case $\frac{\bar{B}(A_n)}{A_n} < \bar{b} - \varepsilon$ can be dealt with similarly.

Therefore, let us prove \eqref{limitBtilde} in the following. To do so, start noticing that from \eqref{RBAform} we have
\[
R\left(bA,A\right)  = \frac{\beta_0}{bA \mu}\sum_{m=0}^\infty \frac{g_m(A)}{m!} \int_0^{\frac{bA}{\mu}}x^m e^{-x}dx = R_0 \frac{1}{\frac{bA}{\mu}}\sum_{m=0}^\infty \frac{g_m(A)}{g_0 m!} \gamma\left(m+1,\frac{bA}{\mu}\right),
\]
where $\gamma$ denotes the lower gamma function, $g_0:=g(0)$ and $R_0=\beta_0 g_0/\mu^2$ is the value of $R_0$ according to \eqref{R0eq} when $\beta(x)=\beta_0 x$.

Then, according to Proposition \ref{propRiemanLimit} in Appendix \ref{AppendixAuxProp} (with $G(x)=g(b\mu^{-1} x)/g_0$, and recall that $g_m(A)=g(m/A)$), we have that
\begin{equation}\label{RiemLimit}
\begin{aligned}
l & = \lim_{A\rightarrow \infty} R\left(bA,A\right)=R_0 \lim_{A\rightarrow \infty} \frac{1}{\frac{b}{\mu} A} \sum_{m=0}^\infty \frac{g_m(A)}{g_0 m!} \gamma\left(m+1,\frac{b}{\mu} A\right) \\
& = R_0 \lim_{A\rightarrow \infty} \frac{1}{\frac{b}{\mu} A} \sum_{m=0}^\infty G\left(\frac{m}{\frac{b}{\mu}A}\right)\frac{1}{m!} \gamma\left(m+1,\frac{b}{\mu} A\right) = R_0 \frac{1}{\frac{b}{\mu}}\int_0^{\frac{b}{\mu}} \frac{g(x)}{g_0}dx.
\end{aligned}
\end{equation}

In particular notice that $l$, as a function of $b$, is decreasing because $g$ is decreasing. To complete the proof it is then enough to show that $l=1$ if $b=\bar{b}$. To see this notice that from (\ref{rational_size1}) we have (using $\beta(x)=\beta_0 x$)
\begin{equation}\label{sbarForm}
\begin{aligned}
&\beta_0 \int_0^\infty \bar{s}_\text{det}(a;b)e^{-\mu a}da =\beta_0 \int_0^\infty \int_0^a g\left( \frac{b e^{-\mu \tau}}{\mu} \right)d\tau e^{-\mu a}da \\ &\qquad = \beta_0 \int_0^\infty \int_\tau^\infty e^{-\mu a}da\; g\left( \frac{b e^{-\mu \tau}}{\mu}\right)d\tau = \frac{\beta_0}{\mu^2} \int_0^{\frac{b}{\mu}} g(x)\frac{\mu}{b} dx = R_0\frac{1}{\frac{b}{\mu}}\int_0^{\frac{b}{\mu}} \frac{g(x)}{g_0} dx,
\end{aligned}
\end{equation}
so that from \eqref{RiemLimit} and \eqref{sbarForm} it follows that
\begin{equation}\label{leqRdet}
l=\beta_0 \int_0^\infty \bar{s}_\text{det}(a;b)e^{-\mu a}da,
\end{equation}
and clearly $l=1$ if $b=\bar{b}$ since $\bar{b}$ satisfies
\eqref{bbarDet}.

\end{proof}

\section{Numerical comparison between the deterministic and stochastic models}\label{numericalSection}

In Figure \ref{figTeoVsSim} we compare the deterministic and the stochastic versions of the hierarchic size-structured model when the birth rate is $\beta(s)=\beta_0 s$ and the growth rate has the form
\[
g(z)=\frac{g_0}{1+\frac{z}{z_0}}.
\]
Each row is associated to different $z_0$ values: $z_0=5$ in A, $z_0=1$ in B and $z_0=0.2$ in C. The area parameter used in the stochastic model is $A=1$ (so that $\bar{B}/A=\bar{B}$ and $\bar{B}$ compares directly to $\bar{b}$ as explained in Section \ref{sthocSection}). In the first column we show the expected size of the individuals according to their age. The grey points correspond to individual pairs (age, size) {found in the population throughout a trajectory of the stochastic process simulated according to Figure \ref{programDiagram} (the same trajectory is used in the first 3 columns of each row)}. The purple and orange lines correspond to the theoretical curves according to the deterministic and the stochastic versions respectively
(equations {\eqref{rational_size1}} and \eqref{eqSizeAtAge1}).
As expected the orange line is in agreement with the simulation. It is surprising, however, that for not too small values of $z_0$, the results derived from the deterministic versions are also in accordance with the simulation {(Figure \ref{figTeoVsSim}, first row)}. The second column shows the normalised histogram of the individual sizes of a numerical simulation once the trajectory ``has reached'' the quasi-invariant distribution. As in the first column we observe that the shape of the size distribution in A and B is well predicted by both the analytical results derived from the stochastic version, i.e. equation \eqref{sizeDistSth}, and the analytical results derived from the deterministic version, i.e. {\eqref{sizeDistDet}}.
Notice that whereas we would expect the simulations to agree with \eqref{sizeDistSth}, the close agreement between the simulation and {\eqref{sizeDistDet}} is relatively unexpected since, a priori, there is no formal link between the simulations and the deterministic model when $A$ is fixed. In the light of Theorem \ref{ThmAsyLimit}, however, we can say that $A\geq 1$ is large enough for the sets of parameters used in A and B. The 4th column shows how $\bar{B}$ {and $\bar{b}$} vary as a function of $R_0$ {(using equations \eqref{BEqSth1} and \eqref{bbarDet} respectively)}. The results suggest that the quasi-stationary birth rate is only well predicted by the deterministic version if $z_0$ is large enough compared to the area parameter $A$. In general it seems that the deterministic versions underestimate the quasi-stationary birth rate of the stochastic version. In the 3rd column a simulated trajectory of the stochastic version is shown (where the total population is plotted). In Figure \ref{HeuVsSim} we check the relation between the approximated quasi-stationary birth rate $\bar{B}(A)$ and empirical approximations of the real quasi-stationary birth $\bar{B}_*(A)$ (see Remark \ref{remarkBAprox} for more details).

\begin{figure}[h]
\includegraphics[width=\textwidth]{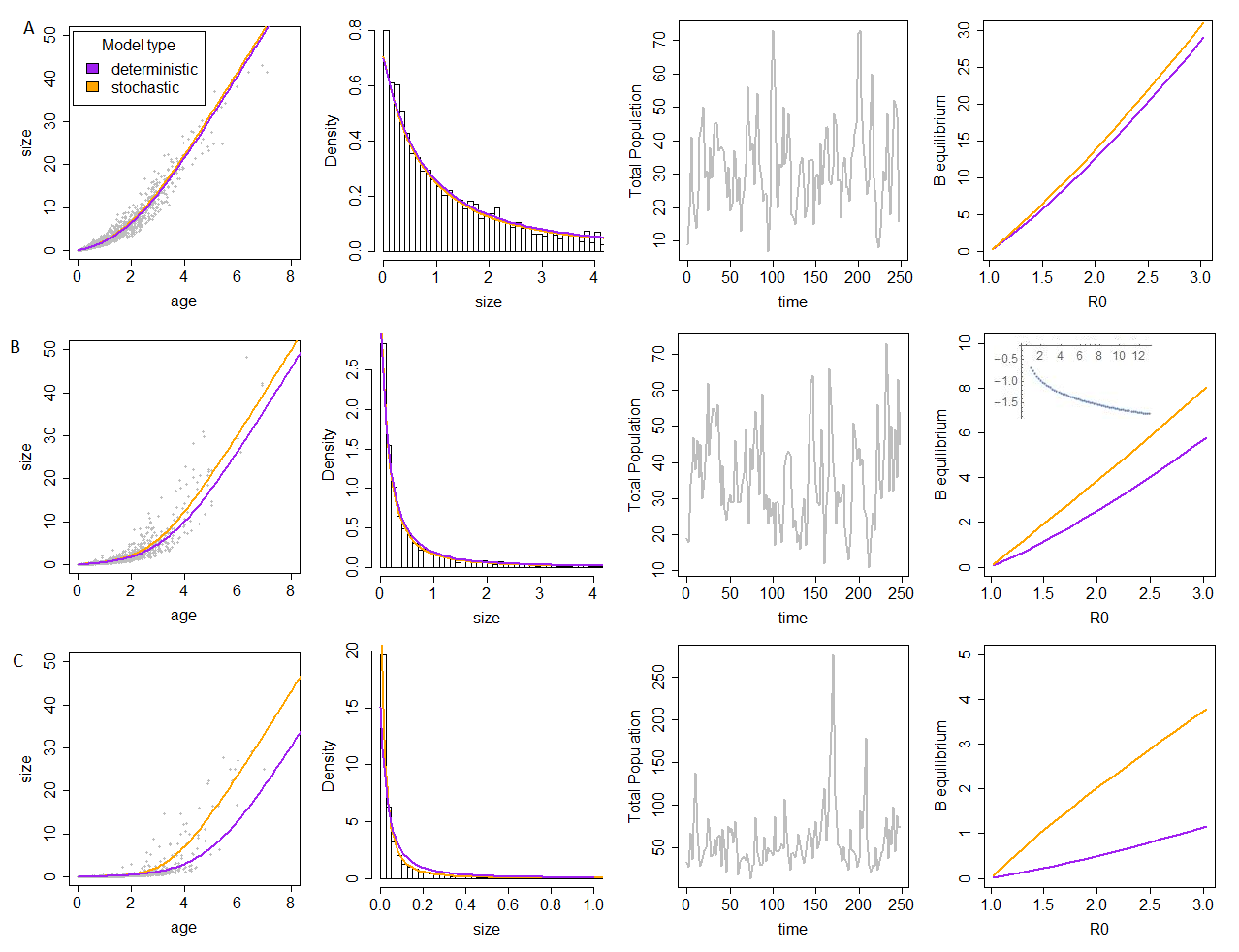}
\caption{\label{figTeoVsSim} 
Comparison between the deterministic and stochastic versions of the hierarchical size-structured model (with $A=1$) for growth rates $g(z)=g_0/(1+z/z_0)$ and birth rates $\beta(s)=\beta_0 s$. The parameters used are the following:  $\mu=1$, $g_0=10$ and $\beta_0$ is such that $\bar{b}=30$ in columns 1, 2 and 3, while in column 4 $\beta_0$ varies so that $R_0$ ranges from 1 to 3. Each row is associated to different $z_0$ values: $z_0=5$ in A, $z_0=1$ in B and $z_0=0.2$ in C. {Results from simulations of the stochastic model are coloured in gray, whereas results from the analytical formulas associated to the deterministic and the stochastic models are coloured in purple and orange respectively. The inset plot of row B and column 4 shows the ratio $(\bar{B}-\bar{b})/\bar{b}$ as a function of $R_0$ in a $\log_{10}$-$\log_{10}$ scale. For more details see Section 4.}}
\end{figure}

The deterministic model corresponds to the dynamics of the expected population distribution of the stochastic model when  individuals do not interact (notice that this assumption is equivalent to say that the functional ecological responses are linear, so that the deterministic model obtained under this assumption is necessarily linear). Therefore, to understand the reasons behind the similarity (or discrepancy) between the deterministic and stochastic models in general one can ask how far the dynamics of the two models are from being linear.

Proceeding in this way we explain why the two models are more similar for larger values of $z_0$. Indeed, it turns out that in these cases the non-linearity of the model is, in some sense, softened. This can be seen by comparing the graphs of $g(z)=g_0/(1+z/z_0)$ for different $z_0$ values. When $z_0$ is small, the growth rate of individuals is notably slower under the presence of larger individuals, which implies that the individual in the top of the hierarchy grows much faster than the other ones. Instead, as $z_0$ increases, the growth rate of an individual is less affected by the presence of larger individuals in the population (only under a huge amount of larger competitors an individual experiences a significant slow down of its growth rate, but notice that for this to happen the population must have a lot of individuals). This makes the growth rate of all individuals to be relatively independent (and thus the system is closer to be linear).

Interestingly, as it can be shown in the fourth column of Figure \ref{figTeoVsSim}, the critical value of $\beta_0$ above which the deterministic model has a non-trivial stationary distribution coincides with the threshold that determines the existence of a quasi-stationary distribution in the stochastic model.
Such a relation is also explained by noticing that the model at this regime is close to be linear. Indeed, when the parameters are such that the population birth rate ($\bar{b}$ and $\bar{B}$ respectively in the deterministic and stochastic models) tends to be close to $0$, then the population tends to be composed by a single individual whose dynamics is, obviously, independent of other individuals.

When comparing the two models one would expect the similarity between them to increase as the population size at the quasi-stationary distribution tends to infinity. When looking at the fourth column of Figure \ref{figTeoVsSim}, one observes rather the contrary. This could be due to the intrinsic low population found at the top of the hierarchy: no matter how many individuals are there in the population, only a few are among the largest. Hence, the dynamics of these individuals is inherently stochastic in the sense that it is not possible to average the effect of ``infinitely many'' other individuals when the population tends to infinity (as it can be done in well mixed population models). It is even unclear if the relative error between the two versions tends to zero at this limit. Indeed, using equations (5.16) in \cite{BCDF2023} and \eqref{tildeBeq} {(which are explicit formulas for $\bar{b}$ and $\bar{B}$ in the case $z_0=1$), in the inset plot of row B column 4} we show how the ratio $(\bar{B}-\bar{b})/\bar{b}$ depends on $R_0$ in a $\log_{10}$-$\log_{10}$ scale. Although the relative error between $\bar{b}$ and $\bar{B}$ decreases as $R_0$ increases, this decrease slows down as $R_0$ increases.

\begin{figure}[h]
\includegraphics[width=\textwidth]{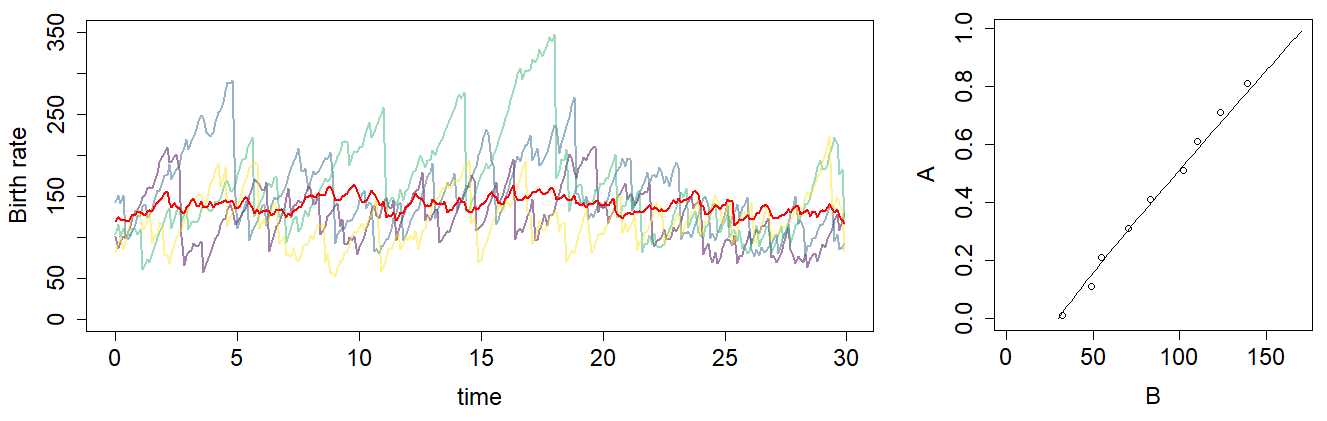}
\caption{\label{HeuVsSim} 
Comparison between the heuristic computation of the quasistationary birth rate $\bar{B}$ according to formula (\ref{BEqSth1}) and the quasistationary birth rate obtained simulating the process. In the left plot four non-extinguishing trajectories in $t\in(0,30)$ of the process are drawn for. The red curve is the average of 10 non-extinguishing trajectories in $t\in(0,30)$. Parameters are as in Figure \ref{figTeoVsSim} but with $\beta_0=3$ and $A=0.81$. The dot with coordinate $A=0.81$ in the right plot is obtained as the mean value taken by the red curve of the left plot. The other dots are obtained analogously by simulating the process with $A\in\{0.01,0.11,0.21,...,0.71\}$. The curve in the right plot is given by $R(B,A)=1$, which determines $\bar{B}$ as a function of $A$.}
\end{figure}

The specific assumption on the birth rate could be another biological phenomena underpinning this discrepancy. Since we consider that larger individuals give rise to more offspring and, as commented above, the larger the individual the less deterministic is its behaviour, the population birth rate of the stochastic model is expected to exhibit a large level of volatility. This is empirically observed in the trajectories shown in column 3 of Figure \ref{figTeoVsSim}: the volatility of total population (a readout of birth rate since the mortality rate is constant) seems to be high for low values of $z_0$ and tends to decrease as $z_0$ increases. This fact could explain why the stationary population birth rate of the deterministic model subestimates the quasi-stationary population birth rate of the stochastic model. Notice that this is a tricky question since the population birth rate is determined by two elements, namely the total population and the individual sizes, that are inversely correlated from the dynamic point of view. Indeed, if the total population is large most individual sizes are expected to be small,  whereas if the total population is small some individual sizes are expected to be large. Taking this into account, we think that the fast growth (and hence the large sizes they reach) of individuals at the top of the hierarchy in the stochastic model give rise to more offspring compared to the offspring produced by all individuals represented by the tail of the stationary distribution of the deterministic model.

\section{Conclusions}

Structured population models have attracted a vast amount of interest in the past decades.  The traditional formulation of a (deterministic) structured population model employs partial differential equations,  often equipped with non-local boundary conditions. There is also an alternative delay equation formulation of these models,  which has some theoretical advantages when studying stability via linearisation. The question whether the two formulations are equivalent in some dynamical sense has been the subject of the recent papers \cite{2016equivalence, barril2022pls}. An essential difference between them is the state variable of the system: in the pde formulation the state is the population density with respect the structuring variable whereas in the delay formulation it is the birth rate history.

Birth and death processes of structured populations can also be formulated either by considering the population distribution or measure at a given time (see e.g.  \cite{metz2013daphnias,carmona2018some}) or by considering the history of birth events, as we have done in the present work. As our results suggest, the formulation in terms of birth events histories could ease the analysis of the system when all individuals are born with the same individual state.  In this case,  quasi-stationary birth rates are just elements of $\mathbb{R}$,  whereas quasi-stationary population measures are elements of an infinite-dimensional space if the population is structured by a continuous variable (such as age or size). Indeed, it has been such a simplification what allowed us to compare the stationary birth rate of the deterministic model to the quasi-stationary birth rate of the stochastic model and show that the former can be obtained as a large population limit of the latter.  

{The reformulation of the dynamics in terms only of the birth events history is not always possible. For this to happen it is needed that the biological properties of individuals depend only on their own age and on the ages of the other individuals in the population. Such a condition holds in the presented model thanks to the biological assumptions we made. Indeed, as commented in the text, in this model the quality of being larger is equivalent to the quality of being older, so that size hierarchy (meaningful from a mechanistic point of view) reduces to age hierarchy. Then, since the growth rate only depends on the age hierarchy, the size of individuals can be given as a function of their age (and therefore the fertility of individuals, initially given as a function of their size, is effectively determined by their ages). If the model is generalised so that competition is not determined just by the amount of individuals above the hierarchy but it also takes into account the sizes of the competitors, that is if the growth rate of the $i$-th smallest individual takes the form (compare with 
\eqref{growthRateSto})
\[
g\left(\frac{1}{A}\sum_{j=i+1}^{N(t)} w(s_j) \right)
\]
for a given function $s\mapsto w(s)$, then the dynamics cannot be described only in terms of the birth events history (in other words, some other state variable is needed besides the birth events history). Another assumption that plays a role in the reformulation is the fact that all individuals are assumed to be born with the same size. If this is relaxed and individuals are allowed to be born with sizes belonging to some range, then the birth rate is no longer real valued but takes values in some functional space. Although some (determinisitic) population models in which births are distributed can be formulated as a renewal equation for function valued birth rate history \cite{2016equivalence}, it is not clear if this is possible in the system analysed here when births are not concentrated at one state. The problem is that in this case the equivalence between being larger and being older is no longer satisfied, and this prevents us to apply some of the arguments we used. Further work is needed to address this scenario for both the determinisitic and the stochastic versions of the model.
}

\bmhead{Acknowledgements}

This work was partially supported by the research projects MT2017-84214C2-2-P, PID2021-123733NB-I00 and 2021-SGR-00113.  {We thank the reviewers for careful reading of the manuscript.}

\bmhead{Conflict of interest}

The authors declare no conflict of interest.

\def\appendixname{Appendix }
\renewcommand\appendix{\par
  \setcounter{section}{0}%
  \setcounter{subsection}{0}%
  \setcounter{equation}{0}
  \gdef\thefigure{\@Alph\c@section.\arabic{figure}}%
  \gdef\thetable{\@Alph\c@section.\arabic{table}}%
  \gdef\thesection{\appendixname~\@Alph\c@section}%
  \@addtoreset{equation}{section}%
  \gdef\theequation{\@Alph\c@section.\arabic{equation}}%
  \addtocontents{toc}{\string\let\string\numberline\string\tmptocnumberline}{}{}
}

\begin{appendices}

\makeatletter
\def\@seccntformat#1{\@ifundefined{#1@cntformat}%
   {\csname the#1\endcsname\space}
   {\csname #1@cntformat\endcsname}}
\newcommand\section@cntformat{\appendixname\thesection.\space} 
\makeatother
\renewcommand{\thesection}{S\arabic{section}}

\section{Derivation of equation (\ref{eqSizeAtAge1})}\label{apendixSizeAtAge}

Using (\ref{sAgek}) in (\ref{sumSAgek}), the function $\bar{s}$ can be rewritten as
$$
\bar{s}(\tau)=\sum_{m=0}^\infty g_m \sum_{k=m}^\infty p_k l_m^k(\tau) .
$$
Therefore to show that $\bar{s}(\tau)$ satisfies (\ref{eqSizeAtAge1}) we prove that (below we use $\lambda$ instead of $\bar{B}$ to ease readability) 
\begin{equation}\label{apSumPkL}
\sum_{k=m}^\infty p_k l_m^k(\tau) = \frac{1}{\mu}\frac{1}{m!}
	\left(
	\Gamma(m,\frac{\lambda}{\mu}e^{-\mu\tau})
	-
	\Gamma(m,\frac{\lambda}{\mu})
	\right).
\end{equation}
In order to accomplish that, we first give the law of $Y_j^k$ and then, in Proposition \ref{propAplToGam} below, the stated result is proven. Recall that the law of $Y_j^k$ has to be given only when $k\geq 1$ and $j\in\{1,\dots,k\}$, since $Y_0^k\equiv 0$ for $k>0$ and $Y_0^0\equiv\infty$.

\begin{proposition} The probability density function of $Y_j^k$ defined in (\ref{YjkDef}) for $k\geq 1$ and $j\in\{1,\dots,k\}$ is given by
$$
f_{Y_j^k}(y) = \frac{1}{(j-1)!}\frac{k!}{(k-j)!}\mu e^{-(k-(j-1))\mu y}(1-e^{-\mu y})^{j-1}\mathds{1}_+(y)
$$ 
\end{proposition}
\begin{proof}
First notice that $Y_1^k$ follows an exponential distribution with rate $k\mu$. Indeed, the first death occurs when one of the $k$ individuals dies. Since the survival time of each individual is exponentially distributed with parameter $\mu$, the time at which the first death occurs is exponentially distributed with rate $k\mu$ (because it is the minimum of $k$ independent random variables exponentially distributed with rate $\mu$). In particular, the statement of the theorem is true for $j=1$ since
$$
f_{Y_1^k}(y)=k\mu e^{-k\mu y}\mathds{1}_+(y).
$$
To show that the formula is also valid for $j>1$ we proceed by induction on $j$. By definition, $Y_{j+1}^k$ can be written as $Y_j^k+X_{j+1}^k$, where $X_{j+1}^k$ is the duration of the lapse between the times when the $j$ and the $j+1$ deaths occur. As argued before, since after the $j$ death there are $k-j$ individuals in the population, $X_{j+1}^k$ is exponentially distributed with rate $(k-j)\mu$. Moreover, $Y_j^k$ is independent of $X_{j+1}^k$. Therefore,
$$
f_{Y_{j+1}^k}(y)=f_{Y_j^k+X_{j+1}^k}(y)=\int_{-\infty}^\infty f_{Y_j^k}(x)f_{X_{j+1}^k}(y-x)dx .
$$
Now, using the hypothesis of induction on $f_{Y_j^k}$ and the known density of $X_{j+1}^k$, it follows
$$
f_{Y_{j+1}^k}(y)= \frac{1}{j!}\frac{k!}{(k-j-1)!}\mu e^{-(k-j)\mu y}(1-e^{-\mu y})^{j}\mathds{1}_+(y).
$$
\end{proof}

\begin{proposition}\label{propApSumpkfYk} Let $Y_j^k$ and $p_k$ defined in (\ref{YjkDef}) and (\ref{pkDef}). The following holds:
$$
\sum_{k=m+1}^\infty p_k f_{Y_{k-m}^k}(y) =\left(\frac{\lambda}{\mu}\right)^m\frac{1}{m!}e^{-(m+1)\mu y}\lambda e^{-\frac{\lambda}{\mu}e^{-\mu y}}\mathds{1}_+(y)
$$ 
\end{proposition}

\begin{proof}
For $y<0$ it is clearly true and for $y\geq 0,$
\begin{small}
\[
\begin{split}
\sum_{k=m+1}^\infty p_k f_{Y_{k-m}^k}(y) & = 
\sum_{k=m+1}^\infty
\left(\frac{\lambda}{\mu}\right)^k \frac{1}{k!}e^{-\lambda/\mu}
\frac{1}{(k-m-1)!}\frac{k!}{m!}\mu e^{-(m+1)\mu y}(1-e^{-\mu y})^{k-m-1} \\
 & =
\left(
\frac{\lambda}{\mu}
\right)^{m+1}
\frac{1}{m!}
e^{-\frac{\lambda}{\mu}}
e^{-(m+1)\mu y}\mu
\sum_{k=0}^\infty \left(
\frac{\lambda}{\mu}\right)^k
\frac{1}{k!}
\left(1-e^{-\mu y}\right)^k\\
 & = 
\left(
\frac{\lambda}{\mu}
\right)^{m+1}
\frac{1}{m!}
e^{-\frac{\lambda}{\mu}}
e^{-(m+1)\mu y}\mu
e^{\frac{\lambda}{\mu}(1-e^{-\mu y})}\\
& = 
\left(\frac{\lambda}{\mu}\right)^m\frac{1}{m!}e^{-(m+1)\mu y}\lambda e^{-\frac{\lambda}{\mu}e^{-\mu y}}
\end{split}
\]
\end{small}

\end{proof}

\begin{proposition}\label{propAplToGam} Equality (\ref{apSumPkL}) holds true.
\end{proposition}

\begin{proof}
First consider $m>0$. In this case it makes sense to consider the duration of the lapse between the $k-m$ and the $k-m+1$ deaths in a population of $k$ individuals, which is the random variable
$$
X_{k-m+1}^k := Y_{k-m+1}^k - Y_{k-m}^k.
$$
Notice that $X_{k-m+1}^k$ is independent of $Y_{k-m}^k$. Moreover, since $\mu$ is constant, the probability density function of $X_{k-m+1}^k$ is
\begin{equation}\label{exp}
f_{X_{k-m+1}^k}(x)=m\mu e^{-m\mu x}
\end{equation}
because when the $k-m$ death takes place there remain $m$ individuals in the population. Observe that $f_{X_{k-m+1}^k}(x)$ does not depend on index $k$. Taking this into account, the conditional expectations in the definition of $l_m^k(\tau)$ can be written as:
$$
E\left(Y_{k-m+1}^k-Y_{k-m}^k | Y_{k-m+1}^k < \tau\right)= E\left(X_{k-m+1}^k | Y_{k-m}^k + X_{k-m+1}^k < \tau\right)
$$
and
\begin{small}
$$
E\left(\tau-Y_{k-m}^k | Y_{k-m}^k < \tau, Y_{k-m+1}^k > \tau\right)= E\left(\tau-Y_{k-m}^k | Y_{k-m}^k < \tau, Y_{k-m}^k + X_{k-m+1}^k > \tau\right),
$$
\end{small}
\noindent so that
\[
\begin{aligned}
 E\left(Y_{k-m+1}^k-Y_{k-m}^k | Y_{k-m+1}^k < \tau\right)& P\left(Y_{k-m+1}^k < \tau\right)\\
\qquad \qquad & =\int_0^\tau\int_0^{\tau-y} x f_{X_{k-m+1}^k}(x)f_{Y_{k-m}^k}(y) dx dy
\end{aligned}
\]
and
\begin{equation*}
\begin{aligned}
E\left(\tau-Y_{k-m}^k | Y_{k-m}^k < \tau, Y_{k-m+1}^k > \tau\right)& P\left(Y_{k-m}^k < \tau, Y_{k-m+1}^k > \tau\right) \\ 
& = \int_0^\tau\int_{\tau-y}^\infty (\tau-y) f_{X_{k-m+1}^k}(x)f_{Y_{k-m}^k}(y) dx dy.
\end{aligned}
\end{equation*}
By computing the inner integrals with respect to $x$ using \eqref{lapse} and \eqref{exp}, one obtains
$$
l_m^k(\tau)=\frac{1}{m\mu}\int_0^\tau\left(1-e^{-m\mu(\tau-y)}\right)f_{Y_{k-m}^k}(y)dy.
$$
The case $k=m$ is special because $l_k^k(\tau)$ depends on $Y_0^k\equiv 0$ (since $k\geq m>0$), which implies
$$
l_m^m(\tau)=\frac{1}{m\mu}\left(1-e^{-m\mu \tau}\right).
$$
Then, the sum in (\ref{apSumPkL}) becomes
$$
p_m l_m^m(\tau)+\frac{1}{m\mu}\int_0^\tau\left(1-e^{-m\mu(\tau-y)}\right)\sum_{k=m+1}^\infty p_k f_{Y_{k-m}^k}(y)dy,
$$
and using Proposition \ref{propApSumpkfYk} to replace the sum in the integral, the whole expression can be rearranged, after some calculations (using a change of variables $\frac{\lambda}{\mu}e^{-\mu y}=z$), so that it gives
\begin{equation}\label{formulaGammas}
\frac{1}{\mu}\frac{1}{m!}
	\left(
	\Gamma\left(m,\frac{\lambda}{\mu}e^{-\mu\tau}\right)
	-
	\Gamma\left(m,\frac{\lambda}{\mu}\right)
	\right).
\end{equation}
When $m=0$, the terms $l_0^k(\tau)$ in \eqref{lapse} satisfy
$$
l_0^k(\tau)=E\left(\tau-Y_k^k | Y_k^k<\tau\right)P\left(Y_k^k < \tau\right)=\int_0^\tau (\tau-y)f_{Y_k^k}(y)dy,
$$
with the particular case $k=0$ given by $l_0^0(\tau)=\tau$. Then, proceeding as before, the left hand side in (\ref{apSumPkL}) becomes
$$
p_0 l_0^0(\tau)+\int_0^\tau(\tau-y)\sum_{k=1}^\infty p_k f_{Y_k^k}(y)dy,
$$
and using the formulas for $p_k$ and Proposition \ref{propApSumpkfYk}, the above expression can be written as (\ref{formulaGammas}) with $m=0$.

\end{proof}

\section{Explicit computation of $\bar{B}$ with $g(z)=g_0/(1+z)$ and $A=1$}

When the individual growth rate is $g(z)=g_0/(1+z)$ and $A=1$, we have that $g_m = \frac{g_0}{1+m}$ and \eqref{eqSizeAtAge1} takes a more closed form (using $\lambda$ instead of $\bar{B}$)
$$
\begin{aligned}
\bar{s}(\tau) & =  \frac{g_0}{\mu} \int_{\frac{\lambda}{\mu} e^{-\mu \tau}}^{\frac{\lambda}{\mu}} \sum_{m=0}^\infty \frac{t^{m-1}}{(m+1)!} e^{-t} dt \\
& = \frac{g_0}{\mu} \int_{\frac{\lambda}{\mu} e^{-\mu \tau}}^{\frac{\lambda}{\mu}} \frac{1-e^{-t}}{t^2} dt = \frac{g_0}{\mu} \left( F \left(\frac{\lambda}{\mu}e^{-\mu \tau}\right) - F\left(\frac{\lambda}{\mu} \right) \right),
\end{aligned}
$$
where $F(t) := \frac{1-e^{-t}}{t} + \Gamma(0,t),$ is a primitive of the function $ -\frac{1-e^{-t}}{t^2}$.\\
We are now able to obtain a transcendental equation for $\bar{B}.$ Indeed, the right hand side of \eqref{BEqSth1} yields
$$
\begin{aligned}
\beta_0 \int_0^\infty \bar{s}(\tau;\lambda)e^{-\mu \tau} d\tau & = \frac{\beta_0 g_0}{\mu} \int_0^{\infty} \int_{\frac{\lambda}{\mu} e^{-\mu \tau}}^{\frac{\lambda}{\mu}} \frac{1-e^{-t}}{t^2} dt e^{-\mu \tau} d \tau \\
& =  \frac{\beta_0 g_0}{\mu \lambda} \Big( \gamma + \ln(\lambda/\mu) + \Gamma(0, \lambda/\mu) \Big),
\end{aligned}
$$
where we performed a change in the integration order and where $\gamma$ is the Euler's constant.\\
Since $R_0 = \frac{\beta_0g_0}{\mu^2}$ if $\beta(s)=\beta_0 s$ (see (\ref{R0eq})), assuming this value is larger than $1$, we can set the ``stochastic" analogous of (5.16) in \cite{BCDF2023} as $\bar{B}$ being the unique solution of the equation
\begin{equation}\label{tildeBeq}
\frac{\gamma + \ln \big( \frac{\bar{B}}{\mu}\big) + \Gamma \big(0,\frac{\bar{B}}{\mu} \big) }{\frac{\bar{B}}{\mu}} = \frac{1}{R_0}.
\end{equation}
(Notice that the function $\frac{\gamma + \ln(x) + \Gamma(0,x)}{x} $ for $x>0$ is strictly decreasing with limit $1$ at $0$ and limit $0$ at infinity.)\\
\\
On the other hand, the derivative of $\bar{s}(\tau)$ reduces to 
\begin{equation}\label{sizederivative}
\bar{s}'(\tau) = -\frac{g_0}{\mu} \frac{ 1-\exp\big(-\frac{\lambda}{\mu} e^{-\mu \tau} \big)}{\big(\frac{\lambda}{\mu}\big)^2 e^{-2 \mu \tau}}\big( -\lambda e^{-\mu\tau} \big) = g_0 \frac{ 1-\exp\big(-\frac{\lambda}{\mu} e^{-\mu \tau} \big)}{\frac{\lambda}{\mu} e^{-\mu \tau}}.
\end{equation}
We now can write the inverse function of $\bar{s}(\tau)$ and also $e^{-\mu \bar{s}^{-1}(x)}$, in terms of the inverse function of $F$ as
\begin{equation}\label{inverse}
e^{-\mu \bar{s}^{-1}(x)} = \frac{\mu}{\lambda} F^{-1}\left(\frac{\mu}{g_0} x +F\left(\frac{\lambda}{\mu}\right)\right)\,.
\end{equation}
Plugging \eqref{inverse} and \eqref{sizederivative} in \eqref{sizeDistSth} and replacing $\lambda$ with $\bar{B}$ we end up with a rather implicit expression for the size distribution
\begin{equation}\label{sizedensitystoch}
\bar{U}(x) = \bar{B} \frac{e^{-\mu \bar{s}^{-1}(x)}}{\bar{s}'(\bar{s}^{-1}(x))} = \frac{\mu}{g_0 } \frac{\left( F^{-1}\left(\frac{\mu}{g_0} x +F\left(\frac{\bar{B}}{\mu}\right)\right) \right)^2}{1-\exp\left(-F^{-1}\left(\frac{\mu}{g_0} x +F\left(\frac{\tilde{B}}{\mu}\right)\right) \right)} \,.
\end{equation}

\section{Auxiliary propositions}\label{AppendixAuxProp}

\begin{proposition}\label{propRiemanLimit}
%

Let $G:\mathbb{R_+}\rightarrow \mathbb{R}$ be a bounded function which is Riemann integrable on $[0,1]$. Then
\[
\lim_{x\rightarrow\infty} \frac{1}{x} \sum_{m=0}^{\infty} G\left(\frac{m}{x}\right) \frac{\gamma(m+1,x)}{m!} = \int_0^1 G(y) dy,
\]
where $\gamma$ is the lower incomplete gamma function.

\end{proposition}
\begin{proof}

Notice that for $k$ a natural number we have,
\begin{footnotesize}
		\[
		\begin{aligned}
		\frac{1}{k} \sum_{m=0}^{\infty} G\left(\frac{m}{k}\right) & \frac{\gamma(m+1,k)}{m!}  \\
		 & = \frac{1}{k} \sum_{m=0}^{k-1} G\left(\frac{m}{k}\right) - \frac{1}{k} \sum_{m=0}^{k-1} G\left(\frac{m}{k}\right) \left(1- \frac{\gamma(m+1,k)}{m!} \right) + \frac{1}{k} \sum_{m=k}^{\infty} G\left(\frac{m}{k}\right) \frac{\gamma(m+1, k)}{m!}\\
		& = \frac{1}{k} \sum_{m=0}^{k-1} G\left(\frac{m}{k}\right) - \frac{1}{k} \sum_{m=0}^{k-1} G\left(\frac{m}{k}\right) \frac{\Gamma(m+1,k)}{m!}  + \frac{1}{k} \sum_{m=k}^{\infty} G\left(\frac{m}{k}\right) \frac{\gamma(m+1,k)}{m!},
		\end{aligned}
		\]
\end{footnotesize}

\noindent where $\Gamma$ denotes the upper incomplete gamma function. Now, since
	\[
	\lim_{k\rightarrow\infty} \frac{1}{k} \sum_{m=0}^{k-1} G\left(\frac{m}{k}\right) = \int_0^1 G(y)dy,
	\]
	to conclude the proof in this case it is enough to check that
	\begin{equation}\label{limitTails0}
	\lim_{k\rightarrow\infty} \frac{1}{k} \sum_{m=0}^{k-1} G\left(\frac{m}{k}\right) \frac{\Gamma(m+1,k)}{m!}=0 \quad\text{and} \quad \lim_{k\rightarrow\infty} \frac{1}{k} \sum_{m=k}^{\infty} G\left(\frac{m}{k}\right) \frac{\gamma(m+1,k)}{m!}=0.
	\end{equation}
	Indeed, this follows from the fact that
	\[
	\begin{aligned}
	\frac{1}{k}
	\sum_{m=0}^{k-1} \frac{\Gamma(m+1,k)}{m!}
	&=\frac{1}{k}\sum_{m=0}^{k-1} e^{-k}\sum_{j=0}^m \frac{k^j}{j!} 
	= \frac{e^{-k}}{k}\sum_{j=0}^{k-1}\sum_{m=j}^{k-1} \frac{k^j}{j!}= \frac{e^{-k}}{k}\sum_{j=0}^{k-1} \frac{k^j}{j!}(k-j)\\
	& = e^{-k}\left(\sum_{j=0}^{k-1}\frac{k^j}{j!}-\sum_{j=1}^{k-1}\frac{k^{j-1}}{(j-1)!}\right) = e^{-k} \frac{k^{k-1}}{(k-1)!}.
	\end{aligned}
	\]
	and
	\[
	\begin{aligned}
	\frac{1}{k}
	\sum_{m=k}^{\infty} \frac{\gamma(m+1,k)}{m!}
	&=\frac{1}{k}\sum_{m=k}^{\infty} e^{-k}\sum_{j=m+1}^\infty \frac{k^j}{j!} 
	= \frac{e^{-k}}{k}\sum_{j=k+1}^\infty \sum_{m=k}^{j-1} \frac{k^j}{j!} \\
	& = \frac{e^{-k}}{k}\sum_{j=k+1}^\infty \frac{k^j}{j!}(j-k) = e^{-k}\left(\sum_{j=k}^{\infty}\frac{k^j}{j!}-\sum_{j=k+1}^{\infty}\frac{k^j}{j!}\right) = e^{-k} \frac{k^k}{k!},
	\end{aligned}
	\]
	so that \eqref{limitTails0} for natural $x$ follows using Stirling's formula. \\
	In order to extend the result to real $x$ we can write
$$
\begin{aligned}
\frac{1}{x} \sum_{m=0}^{\infty} & G\left(\frac{m}{x}\right) \frac{\gamma(m+1,x)}{m!} = \frac{1}{[x]}\sum_{m=0}^{\infty} G\left(\frac{m}{[x]}\right) \frac{\gamma(m+1,[x])}{m!} \\
&  + \frac{1}{x} \sum_{m=0}^{\infty} G\left(\frac{m}{x}\right) \frac{\gamma(m+1,x)}{m!} -  \frac{1}{[x]}\sum_{m=0}^{\infty} G\left(\frac{m}{[x]}\right) \frac{\gamma(m+1,[x])}{m!}
\end{aligned}
$$
and
$$
\frac{1}{x} \sum_{m=0}^{\infty} G\left(\frac{m}{x}\right) \frac{\gamma(m+1,x)}{m!} -  \frac{1}{[x]}\sum_{m=0}^{\infty} G\left(\frac{m}{[x]}\right) \frac{\gamma(m+1,[x])}{m!}
$$
$$
= \frac{1}{x}  \sum_{m=0}^{\infty} G\left(\frac{m}{x}\right) \frac{\gamma(m+1,x) - \gamma(m+1,[x])}{m!} + \left(\frac{1}{x}- \frac{1}{[x]} \right) \sum_{m=0}^{\infty} G\left(\frac{m}{x}\right) \frac{\gamma(m+1,[x])}{m!}
$$
$$
+ \frac{1}{[x]}\sum_{m=0}^{\infty} \left(  G\left(\frac{m}{x}\right) - G\left(\frac{m}{[x]}\right) \right) \frac{\gamma(m+1,[x])}{m!} =: I + II + III.
$$
Then we have, for $x>1$,
\begin{small}
\[
\begin{aligned}
0 & \leq I \leq \frac{g_0}{x} \sum_{m=0}^{\infty} \frac{1}{m!} \int_{x-1}^{x} t^{m} e^{-t} dt = \frac{g_0}{x} \int_{x-1}^{x} e^{-t} \sum_{m=0}^{\infty} \frac{1}{m!}  t^{m}  dt = \frac{g_0}{x},\\
|II| & = \frac{x - [x]}{x} \frac{1}{[x]}  \sum_{m=0}^{\infty} \left|G\left(\frac{m}{x}\right)\right| \frac{\gamma(m+1,[x])}{m!} \leq \frac{x - [x]}{x}   \max_{y\in\mathbb{R}_+} |G(y)| \frac{1}{[x]}  \sum_{m=0}^{\infty} \frac{\gamma(m+1,[x])}{m!},\\
0 & \leq III \leq \frac{1}{[x]}\sum_{m=0}^{\infty} G\left(\frac{m}{[x+1]}\right) \frac{\gamma(m+1,[x])}{m!} - \frac{1}{[x]} \sum_{m=0}^{\infty}    G\left(\frac{m}{[x]}\right) \frac{\gamma(m+1,[x])}{m!} 
\end{aligned}
\]
\end{small}

\noindent and the three terms tend to 0 when $x$ goes to $\infty$. This is obvious for the term $I$. For the term $II$ use that $\lim_{x \rightarrow \infty} \frac{x-[x]}{x} = 0$ and that
\[
\lim_{k\rightarrow \infty} \frac{1}{k} \sum_{m=0}^{\infty} \frac{\gamma(m+1,k)}{m!}=1
\]
for $k$ a natural number (which follows from the first part of the proof). Finally for the term $III$ the first part of the proof can be extended to the case when the argument of $G$ is $\frac{m}{k+1}$ instead of $\frac{m}{k}$ to show that
\[
\lim_{k\rightarrow\infty} \frac{1}{k}\sum_{m=0}^{\infty}  G\left(\frac{m}{k+1}\right) \frac{\gamma(m+1,k)}{m!}=\int_0^1 G(y)dy,
\]
since notice that
\[
\begin{aligned}
\lim_{k\rightarrow\infty} \frac{1}{k}\sum_{m=0}^{k-1}  G\left(\frac{m}{k+1}\right) & = \lim_{k\rightarrow\infty} \frac{k+1}{k}\frac{1}{k+1}\sum_{m=0}^{k}  G\left(\frac{m}{k+1}\right) - \frac{k+1}{k}\frac{1}{k+1}  G\left(\frac{k}{k+1}\right) \\
& = \lim_{k\rightarrow\infty} \frac{1}{k}\sum_{m=0}^{k-1} G\left(\frac{m}{k}\right).
\end{aligned}
\]
\end{proof}

\end{appendices}

\end{document}